\documentclass[]{article}
\usepackage[hidelinks,draft=false]{hyperref}
\usepackage{times}
\usepackage{shuffle}
\usepackage{booktabs}
\usepackage{microtype}
\usepackage[final]{graphicx}
\usepackage{geometry}
\usepackage{comment}
\usepackage{mathtools} 
\usepackage{subfig}
\usepackage{float}
\usepackage{amsmath}
\makeatletter
\let\over\@@over
\let\atop\@@atop
\makeatother
\usepackage{amssymb}
\usepackage{mathrsfs}
\usepackage{mhequ}
\usepackage{./styles/mhenvs}
\usepackage{./styles/mhfig}
\usepackage{./styles/mhsymb}
\usepackage{hyperref}

\usepackage{times}
\usepackage{microtype}
\usepackage{amsmath}
\usepackage{amssymb}
\usepackage{tikz}
\usepackage{wasysym}
\usepackage{centernot}

\colorlet{symbols}{blue!90!black}

\def\1{\mathbf{1}}

\colorlet{testcolor}{green!60!black}

\usetikzlibrary{shapes.misc}
\usetikzlibrary{shapes.symbols}
\usetikzlibrary{decorations}
\usetikzlibrary{decorations.markings}

\makeatletter
\def\DeclareSymbol#1#2#3{\expandafter\gdef\csname MH@symb@#1\endcsname{\tikz[baseline=#2,scale=0.15]{#3}}%
\expandafter\gdef\csname MH@symb@#1s\endcsname{\scalebox{0.6}{\tikz[baseline=#2,scale=0.15]{#3}}}}
\def\<#1>{\csname MH@symb@#1\endcsname}
\makeatother

\DeclareSymbol{X}{-2.4}{\node[circle, fill, inner sep=1pt] {};}

\DeclareSymbol{1}{0}{\draw[white] (-.4,0) -- (.4,0); \draw (0,-.2)  -- (0,1.6) node[dot, fill, inner sep=0.7pt] {};}

\DeclareSymbol{sX}{-2.4}{\node[symb, circle, fill, inner sep=0.8pt] {};}

\newop{diag}

\newcommand\minus{%
  \setbox0=\hbox{-}%
  \vcenter{%
    \hrule width\wd0 height \the\fontdimen8\textfont3%
  }%
}

\def\E{\mathbb{E}}
\def\P{\mathbb{P}}
\def\/{\,\rule[-0.25em]{2pt}{1em}\,}

\def\T{\mathbb{T}}

\def\${|\!|\!|}

\def\dd{\mathrm{d}}

\def\R{\mathbb{R}}

\def\N{\mathbb{N}}

\newcommand{\norm}[1]{\left\lVert #1 \right\rVert}

\def\sC{{\mathscr C}}

\def\sT{{\mathscr T}}

\def\fs{{\mathfrak s}}

\usepackage{xcolor}
\newenvironment{red}%
  {\color{red}}%
  {}

\begin{document}  

\title{A priori bounds for the generalized parabolic Anderson model in the full subcritical regime}

\author{Salvador Esquivel, Felix Schwarzfischer}

\maketitle 

\begin{abstract}
    We show a priori bounds for the generalized Parabolic Anderson Model $(\partial_t - \Delta)u = \sigma(u) \diamond \xi$ in finite volume in the full subcritical regime. The approach is based on the observation that, for constant initial data, the time interval over which the solution remains close to its initial condition can be chosen uniformly in the value of the initial condition. A comparison principle is then used to extend this estimate to general bounded initial data. 
\end{abstract}

\tableofcontents

\section{Introduction}
We aim to derive a priori bounds for the solution $u$ to the scalar singular SPDE on the torus $\T^d$ 
\begin{equs}\label{eq:gPAM}
(\partial_t - \Delta + m^2)u &= \sum_{i=1}^\ell \sigma_i (u)  \xi_i \,  \\ 
  u(0, \cdot) &= u_0 \in L^\infty(\T^d)\, , 
\end{equs}
with the noises $\xi_i$ being random variables on some probability space $(\Omega, \CF, \P)$ taking values almost surely in $\sC_\fs^{-\bar{\beta}}(\R \times \T^d)$ \footnote{$\sC_\fs^{-\bar{\beta}}$ denotes the usual Hölder-Besov space with parabolic scaling $\fs = (2,1, \dots, 1)$. },  $\bar{\beta} \in (0,2)$, $\ell \in \N$,  mass $m^2 \in \R_{\geq 0}$ and smooth nonlinearities $\sigma_i$ with sufficiently many bounded derivatives. Due to the low regularity of the driving noises $\xi_i $ when $\bar{\beta}>1$, the products  $\sigma_i(u)  \xi_i$ are not classically well-defined and equation \eref{eq:gPAM} needs to be renormalized to give a meaningful notion of solution. This can be achieved, for example, using the theory of Regularity Structures \cite{hairer2014theory}, paracontrolled calculus \cite{gubinelli2015paracontrolled}, or the flow equation approach \cite{duch2025flow}. Once the local well-posedness of \eref{eq:gPAM} is established in a suitable sense global well-posedness can then be shown by providing a priori bounds on the norm of the solution $u(t)$ preventing blow up in finite time. This problem has garnered a lot of attention in recent years, most notably
\begin{enumerate}
    \item \cite{chandra2026priori} in the framework of regularity structures in the regime $\bar{\beta} \in (0,\tilde{\beta})$ with $\tilde{\beta}\approx 1.13$ on finite volume,  
    \item \cite{shen2026global,shen2026global_onplane} in the framework of paracontrolled calculus in the regime $\bar{\beta} \in (0, \sqrt{5} - 1)$ on the plane $\R^2$,
    \item \cite{bringmann2026global} in the regime $\bar{\beta} \in (0, \frac{3}{2})$ for the dynamical massive sine-Gordon model in finite volume.
\end{enumerate}
In this work, we will employ the theory of regularity structures \cite{hairer2014theory} as the solution theory to give meaning to \eref{eq:gPAM}, whose setup we will now describe in order to state the main result. The strategy, however, does not fundamentally rely on regularity structures and can be adapted to other solution theories. Let the regularity structure $\sT = (A,G, T)$ for \eref{eq:gPAM} be constructed as in \cite[Section 8]{hairer2014theory}, with abstract noise symbols $\{ \Xi_i \}_{i=1, \dots, \ell}$ having homogeneity $|\Xi_i| = -\bar{\beta}$. Let $G_m = (\partial_t - \Delta + m^2)^{-1}$ be the (massive) heat kernel on the torus and for every $\eps>0$  we let $Z^\eps = (\Pi^\eps, \Gamma^\eps)$ be an admissible model (corresponding to a lift of smooth noise $\xi_i^\eps$), renormalized via a strong preparation map in the sense of \cite{bailleul2026locality}. For $\gamma \in (\bar{\beta}, 2)$ we denote $\CU^\eps \in \CD^{\gamma,0}$ as the modelled distribution solving the abstract fixed point equation  
\begin{equs}
\CU^\eps = (\CK^{m, \eps}_{\bar{\gamma}} + R^{m}_\gamma \CR^\eps)(1_{t>0}\sum_{i=1}^\ell \sigma_i(\CU^\eps)\Xi_i ) + G_m u_0 \, , 
\end{equs}
where $\CK^{m, \eps}_{\bar{\gamma}}, R^m_\gamma$ are the decomposed and lifted version of $G_m$ as in \cite[Section 5,6]{hairer2014theory}, with $\bar{\gamma}= \gamma - \bar{\beta} >0 $. 
The regularized and renormalized solution to \eref{eq:gPAM} is then defined by 
$u^\eps = \CR^{Z^\eps}(\CU^\eps)$. By \cite[Theorem 1.1]{bailleul2026locality}, $u^\eps$ solves at every $\eps > 0$ a renormalized equation of the form 
\begin{equ}\label{eq:renormalized_regularized_spde}
(\partial_t - \Delta + m^2)u^\eps(z) = F^\eps(z, u^\eps(z)) =: \sum_{i=1}^\ell \sigma_i(u^\eps) \diamond \xi_i^\eps  \, . 
\end{equ}
The exact form of the nonlinearity $F^\eps$ is not relevant to our approach, its only important feature is that it is a local function of the solution $u^\eps$. Within this context we now state the main result: 
\begin{theorem}\label{thm:apriori_bound}
    Let $u_0 \in L^\infty(\T^d)$ and $u^\eps = \CR^{Z^\eps}(\CU^\eps)$, $\eps > 0$ be a family of regularized and renormalized solutions to the SPDE \eref{eq:gPAM}, then it holds that for any $T \in (0,\infty)$: 
    \begin{enumerate}
        \item If $m^2=0$, there exists a constant $\mathfrak{c} = \mathfrak{c}_{\gamma, \bar{\beta}, \sigma} > 0 $ such that 
        \begin{equ}\label{eq:apriori_gPAM}
        \norm{u^\eps (T)}_\infty  \leq \norm{u_0}_\infty + \mathfrak{c}_{\gamma, \bar{\beta}, \sigma} T   \left(1\vee \$ Z^\eps  \$_{\gamma; [-1, T+2] \times \T^d }^{\left(r + \frac{\gamma}{2-\bar{\beta}} \right) \frac{2}{2-\bar{\beta}} } \right) \, . 
        \end{equ}
        \item If $m^2 > 0$, for any $ p \in \N$ there exist constants $\mathfrak{c} = \mathfrak{c}_{m^2,p}$ and $ \mathfrak{C} = \mathfrak{C}_{\gamma, \bar{\beta}, \sigma, m^2,p}$ such that: 
        \begin{equs}\label{eq:apriori_mSG}
         \E[ \norm{u^\eps(T)}_\infty^p ] \leq e^{-\mathfrak{c}p T} \norm{u_0}^p_\infty + \mathfrak{C} \sup_{n \in \N_0 } \E  \left(1\vee \$ Z^\eps \$^{\left( r + \frac{\gamma}{2-\bar{\beta}} \right) \frac{2p}{2-\bar{\beta}}}_{\gamma;  [n -1, n +3] \times \T^d  }\right) \, . 
    \end{equs}
    \end{enumerate}
\end{theorem}
\begin{remark}
    The dynamical sine-Gordon model is covered by this framework after a Da Prato–Debussche decomposition, upon identifying the appropriate driving noises and nonlinearities. The local well-posedness results of \cite{hairer2016dynamical, chandra2018dynamical} in combination with the uniform in time bound on $p$-th moments of \eref{eq:apriori_mSG} can be used to construct an invariant measure via the Krylov-Bogliubov existence theorem as in \cite[Section 4]{tsatsoulis2018}.  
\end{remark}
The startegy of the proof is the following. For any $a \in \R$ we define $\Phi^\eps(\cdot,  a) : \R_+ \times \T^d \to \R $ as the solution to
    \begin{equs}\label{eq:equation_for_Phi_massive}
    (\partial_t - \Delta +m^2 )\Phi^\eps(z, a) &= m^2 a +  \sum_{i=1}^\ell \sigma_i(\Phi^\eps(z,a))\diamond \xi^\eps_i \\ 
    \Phi^\eps(0, \cdot , a) &= a \, . 
    \end{equs}
 The first step is to show that the time interval for which $\Phi$ stays close to its constant initial condition does not depend on the value of the initial condition:
 \begin{lemma}\label{lem:flow_stability}
    Let $\delta \in (0,1) $, then there exists a time $T_\delta > 0 $, independent of $a \in \R$, such that  the following control on the flow $\Phi$ holds
    \begin{equs}\label{eq:flow_stability}
    \sup_{a \in \R} \sup_{z \in [0,T_\delta] \times \T^d} |\Phi^\eps(z, a) - a| &\leq \delta  \\ \label{eq:flow_stability_derivative}
    1- \delta \leq \sup_{a \in \R} \sup_{z \in [0,T_\delta] \times \T^d} \partial_a \Phi^\eps(z,a) &\leq 1 + \delta  \, .  
    \end{equs}
\end{lemma}
 This is proved by considering the `fluctuating' part $w^\eps_a :=\Phi^\eps(\cdot, a)-a$ of the solution $\Phi^\eps(\cdot, a)$ around $a$ which solves
   \begin{equs}
    (\partial_t - \Delta+m^2) w^\eps_a &= \sum_{i=1}^\ell  \sigma_i(w^\eps_a + a) \diamond \xi =: \sum_{i=1}^\ell \sigma_{i,a}(w^\eps_a) \diamond \xi^\eps  \\ 
    w^\eps_a(0, \cdot) &= 0 \, . 
    \end{equs}
The key observation is that the dependence on the constant initial data $a \in \R$ is absorbed into a shift of the nonlinearity $\sigma_a = \sigma(\cdot + a)$. Since $\norm{\sigma}_{C^k} = \norm{\sigma_a}_{C^k}$ the local existence time for $w^\eps_a$ can be chosen independent of $a \in \R$. 

The second step is to use a comparison principle to transfer this properties of $\Phi^\eps$ to the solution of \eqref{eq:apriori_gPAM} with generic initial condition in $L^\infty(\T^d)$. In the massless case $m^2 = 0$ this reduces to the observation that  \[\Phi^\eps(z, -\norm{u_0}_\infty) \leqslant u^\eps(t) \leqslant \Phi^\eps(z,\norm{u_0}_\infty) \qquad \forall z\in [0,T_\delta]\times \T^d.\]

\begin{proposition}\label{prop:one_time_estimate}
    Fix $\eps>0$ and let $u^\eps$ be a smooth solution to the renormalized SPDE \eref{eq:gPAM} with initial condition $u_0\in L^\infty(\mathbb{T}^d)$. For some $\delta \in (0,1)$, let $T_\delta>0$ be such that \eref{eq:flow_stability}, \eref{eq:flow_stability_derivative} hold. Then:  
        \begin{equ}\label{eq:massive_one_time_estimate}
        \forall t \in [0,T_\delta]: \, \norm{u^\eps(t)}_\infty \leqslant e^{- \frac{m^2}{1+\delta}  t } \norm{u_0}_\infty + \delta  \, . 
        \end{equ}
\end{proposition} 
This comparison argument is available only at the regularized level $\eps>0$, where the renormalised equation \eqref{eq:apriori_mSG} is a local function of the solution. Crucially, however, the resulting estimate is uniform in $\eps$ and therefore survives the limit $\eps\to0$.
The proof of Theorem \ref{thm:apriori_bound} follows then by standard arguments which we include for completness.
\begin{remark}
         Step one considers only constant initial data and hence sidesteps the observation made in \cite[Remark 2.44]{chandra2025rough}, insofar as attempting to absorb the non-constant harmonic extension $(G_tu_0)(x)$ for non-constant $u_0$ runs into the trouble of introducing by the chain rule existence times dependent on $\norm{u_0}_\infty$. 
\end{remark}

\begin{comment}
    The step with the transport term can also be argued with: Instead of directly analyzing the size change of $u$, we consider the size change of $v$, which, conveniently tells us that as $\partial_a \Phi \neq 0$, that at extrema of $v$, we have $(\partial_t - \Delta)v = 0 $. So if we are at a maximum of $v$, $\partial_t v = \Delta v \leq 0 $, and vice versa. Therefore $\norm{v(t, \cdot)}_\infty$ is non-increasing! This size control on $v$ transfers directly to size control of $u$, as the flow $\Phi(t,x, a)$ stays uniformly close to $a$ on the considered time interval, that is $|\Phi(z, a) - a|\leq \delta $ uniformly in $a \in \R$,$z \in [0,h] \times \T^2$.
\end{comment}

\section{Proof of the main theorem}
We will drop the letter $\eps$ below to avoid clutter in the notation, keeping in mind that we always consider the qualitatively smooth pathwise PDE at some regularization $ \eps > 0 $. 

\begin{lemma}\label{lem:control_on_w}
    Let $\delta \in (0,1)$. Then there exists $T_\delta \in (0,1] $ such that for any $a \in \R$, $\CW_a$ is a fixed point of  
    \begin{equ}
    \CW_a = (\CK^{m}_{\bar{\gamma}} + R^m_\gamma \CR)(1_{t> 0} \sum_{i=1}^\ell \sigma_{i,a} (\CW_a) \Xi  ) \, , 
    \end{equ}
    for $\CW_a \in \CD^{\gamma,0}$ on $(0,T_\delta)$, where $\sigma_{i,a}(\CW_a) = \sigma_i(a \1 + \CW_a)$, and such that we have following control on $w_a =\CR \CW_a$: 
    \begin{equs}\label{eq:control_for_w1}
    \sup_{t \in [0,T_\delta ]}\norm{w_a(t,\cdot)}_\infty &\leq \delta \\ \label{eq:control_for_w2}
    \sup_{t \in [0,T_\delta ]} \norm{w_a(t, \cdot) - w_b(t, \cdot)}_\infty & \leq \delta |a-b| \, . 
    \end{equs}
\end{lemma}
\begin{proof}
    We let without loss of generality $\ell = 1$ to keep the notation simple. As $\CW_a$ is a fixed point solution, coming from \cite[Thm. 7.8]{hairer2014theory}, we just need to show that the map 
    \begin{equs}
    \CM_a : \CD^{\gamma,0}_{(0,T)} &\to \CD^{\gamma,0}_{(0,T)} \\ 
    \CW &\mapsto (\CK_{\bar{\gamma}} + R_\gamma \CR)(1_{t> 0} \sigma_a(\CW) \Xi  )
    \end{equs}
    is a well-defined contraction, mapping $\tilde{\delta}$-balls to $\tilde{\delta}$-balls for $T>0$ independent of $a \in \R$, with some $\tilde{\delta} \in (0,\delta \wedge \frac{1}{2})$ specified later. Setting $\zeta = 2-\bar{\beta}$ and recalling $ \$ Z \$_{\gamma; O} := \$ Z \$_{\gamma; [-1, 2] \times \T^d}$ : 
    \begin{enumerate}
        \item To check that it maps balls to balls compute:  
        \begin{equs}
        \$ \CM_a(\CW) \$_{\gamma,0; T} &\leq \$ \CM_a(\CW) \$_{\gamma - \bar{\beta} + 2, -\bar {\beta} + 2  - (2-\bar{\beta}); T } \lesssim T^{\frac{2-\bar{\beta}}{2}} \$ Z \$^r_{\gamma; O} \$ \sigma_a(\CW) \Xi \$_{\gamma - \bar{\beta}, -\bar{\beta}; T }  \\ 
        &\leq T^{\frac{2-\bar{\beta}}{2}}  \$ Z \$^r_{\gamma; O} \left(\sup_{k \leq \frac{\gamma}{\zeta}} \norm{D^k \sigma_a }_\infty  \right) (1 \vee \$Z \$_{\gamma; O})^{\frac{\gamma}{\zeta}}(1 \vee \$ \CW  \$_{\gamma, 0 ; T })^{\frac{\gamma}{\zeta}}\, . 
        \end{equs}
        Here the second inequality is taken from \cite[Theorem 7.1]{hairer2014theory} (with the exponent $r$ of $ \$ Z \$^r$ the one referenced therein), the last inequality from \eref{eq:explicit_bounds_composition}. 
        \item To check that it is a contraction: 
        \begin{equs}
        &\$ \CM_a(\CW) - \CM_a(\CV) \$_{\gamma,0 ; T} \\ \lesssim & T^{\frac{2-\bar{\beta}}{2}} \$ Z \$^r_{\gamma; O} \$ \sigma_a(\CW) - \sigma_a(\CV) \$_{\gamma, 0; T}  \\ 
        \lesssim & T^{\frac{2-\bar{\beta}}{2}} \$ Z \$^r_{\gamma; O} \left(\sup_{k \leq \frac{\gamma}{\zeta} + 1 } \norm{D^k \sigma_a}_\infty  \right)  (1 \vee \$Z \$_{\gamma; \fK})^{\frac{\gamma}{\zeta}}(1 \vee ( \$ \CW \$_{\gamma; \fK} + \$ \CV \$_{\gamma;\fK}) )^{\frac{\gamma}{\zeta}} \$ \CW - \CV  \$_{\gamma; \fK} \, , 
        \end{equs}
        where we again used \cite[Theorem 7.1]{hairer2014theory} and \eref{eq:explicit_bounds_composition_difference}. 
    \end{enumerate} 
    Hence choosing $T_\delta \in (0,1]$ as, say, 
    \begin{equ}\label{eq:local_time_T}
    T_\delta = \epsilon \left( (1\vee \$ Z \$_{\gamma;  O})^{r + \frac{\gamma}{2-\bar{\beta}}} \sup_{k \leq \frac{\gamma}{2-\bar{\beta}} + 1 } \norm{D^k \sigma}_\infty \right)^{ - \frac{2}{2-\bar{\beta}}}
    \end{equ}
    for some $\epsilon = \epsilon(\delta) > 0 $ small enough to absorb all the implicit constants, we get that $\CM_a$ is a strict contraction and maps $\tilde{\delta}$ balls to $\tilde{\delta}$ balls, so $\CW_a$ exists as a solution in $\CD^{\gamma,0}_{(0,T_\delta)}$. 
    This duplication of the arguments of \cite[Thm. 7.8]{hairer2014theory} is to emphasize that the existence time $T_a$ is independent of $a \in \R$ because of 
    \begin{equ}
    \norm{D^k \sigma_a}_\infty  = \norm{D^k \sigma}_\infty \,,  \forall k \in \N_0 \, . 
    \end{equ}
     Coming now to the proofs of \eref{eq:control_for_w1}, \eref{eq:control_for_w2}, one just notes that by definition of the $\$ \cdot \$_{\gamma,0; T}$ norm and the equality $w_a = \langle \1, \CW_a \rangle $ that 
    \begin{equ}
    \sup_{t \in [0,T_\delta]} \norm{w_a(t, \cdot)}_\infty \leq \$ \CW_a \$_{\gamma,0; T_\delta} \leq \tilde{\delta} < \delta  \, . 
    \end{equ}
    Similarly we just employ \eref{eq:bound_difference_of_difference} to arrive at 
    \begin{equs}
    & \$ \CW_a - \CW_b \$_{\gamma, 0; T_\delta} \\ \lesssim & T_\delta^{\frac{2-\bar{\beta}}{2}} \$ Z \$_{\gamma; O}^r \$ \sigma_a(\CW_a) - \sigma_b(\CW_b) \$_{\gamma, 0; T_\delta} \\
    \lesssim & T_\delta^{\frac{2-\bar{\beta}}{2}}  (1 \vee \$ Z \$_{\gamma; O})^{r + \frac{\gamma}{2-\bar{\beta}}}  \left( \left(1 \vee \left(\sup_{k \leq \frac{\gamma}{\zeta} + 1 } \norm{D^k \sigma_a}_\infty  \right) \right) \$ \CW_a - \CW_b \$_{\gamma,0; T_\delta} +  \sup_{k \leq \frac{\gamma}{\zeta}}\norm{D^k(\sigma_a - \sigma_b)}_{\infty} \right) \, . 
    \end{equs}
    So perhaps choosing the prefactor $\epsilon > 0$ for $T_\delta$  even smaller we can absorb $\$ \CW_a - \CW_b \$ $ into the left-hand-side and arrive at 
    \begin{equ}
    \$ \CW_a - \CW_b \$_{\gamma, 0; T} \leq \frac{\tilde{\delta}}{1-\tilde{\delta}} \left(1 \vee \left(\sup_{k \leq \frac{\gamma}{\zeta} + 1 } \norm{D^k \sigma}_\infty  \right) \right)^{-1} \left( \sup_{k \leq \frac{\gamma}{\zeta}}\norm{D^k(\sigma_a - \sigma_b)}_{\infty} \right) \, . 
    \end{equ}
    Finally with the elementary inequality 
    \begin{equ}
    \sup_{k \leq \frac{\gamma}{\zeta}}\norm{D^k(\sigma_a - \sigma_b)}_{\infty} \leq \sup_{k \leq \frac{\gamma}{\zeta} + 1 }\norm{D^k \sigma}_\infty |a-b| \, , 
    \end{equ}
    and choosing $\tilde{\delta} \in (0, \frac{\delta}{1+\delta} \wedge \frac{1}{2})$ yields 
    \begin{equs}
    \sup_{t\in [0,T]} \norm{w_a(t, \cdot) - w_b(t, \cdot)}_\infty &\leq \$ \CW_a - \CW_b \$_{\gamma, 0; T} \leq \delta |a-b| \, . 
    \end{equs}
\end{proof}
As an immediate corollary, these stability properties transfer to the flow $\Phi$.

\begin{proof}[of Lemma \ref{lem:flow_stability}]
    Since $\Phi(z,a) - a = w_a(z)$, \eref{eq:flow_stability} follows directly by \eref{eq:control_for_w1}. Also it is immediate that due to 
    \begin{equ}
    \Phi(z,a) - \Phi(z,b) = (a-b) + w_a(z) - w_b(z) \, , 
    \end{equ}
    inequality \eref{eq:flow_stability_derivative} follows by \eref{eq:control_for_w2}. The fact that $ a \mapsto \Phi(z, a)$ is differentiable follows analogously to \cite[Theorem 3.4.4]{henry2006geometric} (since the non-linearity $F$ is qualitatively smooth and bounded for finite regularizations). 
\end{proof}

\begin{comment}
    The aim is to emphasize that this flow $\Phi$ considered therein is the correct object to work with, except that this is already sufficient to conclude the a priori bound, without having to go through the trouble of introducing another auxiliary field $v$, necessitating properties of  $\Phi$ that go beyond just stability. 
\end{comment}

\begin{proof}[of Proposition \ref{prop:one_time_estimate}]
We drop the $\eps$ subscript, keeping in mind that the estimates are uniform in $\eps>0$. We set
\[ a_{\pm} (t) := \pm \| u_0 \|_{\infty} e^{- \frac{m^2}{1 + \delta} t},
\]
and define the moving barriers
\[ B_{\pm} (t, x) := \Phi ((t, x), a_{\pm} (t)) . \]
Since $a_{\pm}$ is only a function of time,
\begin{eqnarray*}
  \partial_t B_{\pm} (t, x) & = & \partial_t \Phi ((t, x), a_{\pm} (t)) +
  \partial_a \Phi ((t, x), a_{\pm} (t)) a_{\pm}' (t)\\
  \Delta B_{\pm} (t, x) & = & \Delta \Phi ((t, x), a_{\pm} (t)),
\end{eqnarray*}
and so one obtains 
\begin{eqnarray*}
  (\partial_t - \Delta + m^2) B_{\pm} (t, x) & = & [(\partial_t
  - \Delta + m^2) \Phi] ((t, x), a_{\pm} (t)) + \partial_a \Phi ((t, x),
  a_{\pm} (t)) a_{\pm}' (t)\\
  & = & F (t, x, \Phi ((t, x), a_{\pm} (t))) + m^2 a_{\pm} (t) + \partial_a
  \Phi ((t, x), a_{\pm} (t)) a_{\pm}' (t)\\
  & = & F (t, x, B_{\pm} (t, x)) + m^2 a_{\pm} (t) + \partial_a \Phi ((t, x),
  a_{\pm} (t)) a_{\pm}' (t) \, . 
\end{eqnarray*}
Since $a'_{\pm} (t) = - \frac{m^2}{1 + \delta} a_{\pm} (t)$ we conclude that
\[ (\partial_t - \Delta + m^2) B_{\pm} ((t, x), a_{\pm} (t)) - F (t, x,
   B_{\pm} (t, x)) = m^2 a_{\pm} (t)  \left( 1 - \frac{\partial_a \Phi ((t,
   x), a_{\pm} (t))}{1 + \delta} \right) . \]
By \eref{eq:flow_stability_derivative}  the term in brackets is non-negative and since $a_+ \geqslant 0$ and $a_- \leqslant
0$ this implies that $B_+$ is a supersolution and $B_-$ is a subsolution of \eref{eq:gPAM}
. Since at $t = 0$
\[ B_- (0, x) = - \| u_0 \|_{\infty} \leqslant u (0, x) \leqslant \| u_0
   \|_{\infty} \leqslant B_+ (0, x), \]
and $F(t,x,\, \cdot \,)$ is smooth, the comparison principle for semilinear parabolic PDEs yields
\[ B_- (t, x) \leqslant u (t, x) \leqslant B_+ (t, x) \qquad \forall t \in [0,
   T_{\delta}] . \]
Finally, by \eref{eq:flow_stability} one has \[|B_\pm(t,x)-a_\pm(t)|=|\Phi((t,x),a_\pm(t))-a_\pm(t)|\leqslant \delta, \] which implies that 
\[ a_- (t) - \delta \leqslant u (t, x) \leqslant a_+ (t) + \delta \qquad
   \forall t \in [0, T_{\delta}] , \]
and recalling the definition of $a_\pm$ one concludes the result.
\end{proof}
\begin{comment}
\begin{red}
\begin{lemma}[Comparison principle] \label{lem:comparison_principle}
Let $u,v$ be solutions to \eref{eq:gPAM}.  Then for $u_0, v_0 \in L^\infty$ it holds  
\begin{equ}
    u_0 \leq v_0 \, \implies u(t, \cdot) \leq v(t, \cdot) \, . 
\end{equ}
\end{lemma} 

\begin{proof}
    Define the field $w := u - v$. One computes for any spacetime point $z$, utilizing crucially that the renormalized nonlinearity $F$ is a local functional of the solution $u$ that 
    \begin{equs}
        (\partial_t - \Delta) w(z) &= F(z, u(z)) - F(z,v(z))  = \int_0^1 \frac{\dd }{\dd \theta }(F(z, v(z) + \theta(u(z)- v(z)) )  \dd \theta \\ 
        &= w(z) \int_0^1 \partial_{\mathsf{u}} F (z, v(z) + \theta(u(z) - v(z))) \dd \theta =: w(z) c(z) \, . 
    \end{equs}
    Here, $\partial_{\mathsf{u}}F$ denotes the derivative of $F$ with respect to the field variable. This newly introduced field $c(z)$ has the property of being smooth and bounded as we work at some finite regularization, hence the classical maximum principle for parabolic PDE yields that as $w_0 = u_0 - v_0 \leq 0 $, that $w(t) \leq 0 $. 
\end{proof}
\end{red}
\end{comment}
Now equipped with the short time estimate \eref{eq:massive_one_time_estimate}, we can close the argument for Theorem \ref{thm:apriori_bound}. 

\begin{proof}[of Theorem \ref{thm:apriori_bound} ]
In the massless case $m^2 = 0 $, we just iterate and reapply Propositon \ref{prop:one_time_estimate}, since the time $T_\delta$ is \emph{independent} of the size of the initial condition. Hence as we iterate at most $\asymp  \frac{T}{T_\delta}$ times, the result follows by the explicit form of $T_\delta$ from \eref{eq:local_time_T}. More precisely, accounting for the restart times and $ \$ Z \$_{\gamma; [s-1, s+2]} \leq \$ Z \$_{\gamma; [-1, T+2]} $ for any $s \in (0,T)$, we arrive at 
    \begin{equ}
    \norm{u^\eps (T, \cdot)}_\infty  \leq \norm{u_0}_\infty + \frac{\delta}{\epsilon} T  \left( (1\vee \$ Z^\eps  \$_{\gamma, [-1, T+2]})^{r + \frac{\gamma}{2-\bar{\beta}}} \sup_{k \leq \frac{\gamma}{2-\bar{\beta}} + 1 } \norm{D^k \sigma}_\infty \right)^{ \frac{2}{2-\bar{\beta}}}\, . 
    \end{equ}
Collecting constants leads to \eref{eq:apriori_gPAM}. In the massive case $m^2 \geq 0$, we follow closely \cite[Section 7]{chandra2026priori}, that is, we prove the estimate in three steps. 
    \begin{enumerate}
        \item For $n \in \N_0$ we aim to get a recursive estimate on the sizes $\norm{u(n)}_\infty$. To this end, we let 
        \begin{equ}
         T_{\delta, n} := \epsilon \left( (1\vee \$ Z \$_{\gamma, [n-1, n+3] \times \T^d })^{r + \frac{\gamma}{2-\bar{\beta}}} \sup_{k \leq \frac{\gamma}{2-\bar{\beta}} + 1 } \norm{D^k \sigma}_\infty \right)^{ - \frac{2}{2-\bar{\beta}}} \, , 
        \end{equ}
        and we define by $N_n$ an upper bound on the number of iterations needed to transfer from time $n$ to $n+1$: 
            \begin{equ}
                N_n := 1 + T^{-1}_{\delta, n} \geq \left\lceil \frac{1}{T_{\delta,n}} \right\rceil  \, . 
            \end{equ}
        We can now iterate \eref{eq:massive_one_time_estimate} on intermediate times $t \in [n,n+1]$ to get 
        \begin{equs}
        \norm{u(n+1, \cdot)}_\infty \leq \delta(1+e^{-\frac{m^2}{1+\delta}T_{\delta, n}} + e^{-\frac{m^2}{1+\delta} 2T_{\delta, n}} \dots) + e^{-\frac{m^2}{1+\delta}} \norm{u(n, \cdot)}_\infty \leq \delta N_n + e^{-\frac{m^2}{1+\delta}} \norm{u(n, \cdot)}_\infty  \, . 
        \end{equs}
        \item Now for $p \in \N$, by taking expectations with respect to the randomness of the driving noise $\xi_i$ we get by the Minkowski inequality
        \begin{equ}
        \left( \E \norm{u(n+1)}_\infty^p \right)^{\frac{1}{p}} \leq \delta \norm{N_n}_{L^p(\Omega, \P)} + e^{-\frac{m^2}{1+\delta}} \left( \E \norm{u(n)}_\infty^p \right)^{\frac{1}{p}} \, . 
        \end{equ}
        Consequently, recursively iterating: 
        \begin{equs}
        \left( \E \norm{u(n)}_\infty^p \right)^{\frac{1}{p}} \leq e^{-\frac{m^2}{1+\delta}n} \norm{u_0}_\infty + \frac{\delta }{1-e^{-\frac{m^2}{1+\delta}}} \sup_{j \leq n-1 } \norm{N_j}_{L^p(\Omega, \P)} \, . 
        \end{equs}
        \item For any intermediate time $t \in [n,n+1]$ we can bound again as in the first step: 
        \begin{equ}
        \norm{u(t)}_\infty \leq \delta N_n + e^{-\frac{m^2}{1+\delta}(t-n)} \norm{u(n)}_\infty \, . 
        \end{equ}
        Hence  in combination with the estimate on integer times we arrive for $t \in \R_{\geq0}$ at 
        \begin{equs}
        \left( \E[ \norm{u(t)}_\infty^p ] \right)^{\frac{1}{p}}  \leq e^{-\frac{m^2}{1+\delta}t} \norm{u_0}_\infty + \frac{\delta}{1-e^{-\frac{m^2}{1+\delta}}} \sup_{n \in \N_0 } \norm{N_n}_{L^p(\Omega, \P)} \, . 
        \end{equs}
        
    \end{enumerate}
     Summarizing the constants yields \eref{eq:apriori_mSG}.
\end{proof}

\appendix

\section{Technical regularity structure lemma}
\begin{lemma}\label{lem:composition_bound}
    Let $F : \R \to \R$ be smooth and $f \in \CD^\gamma(V)$ such that $f: \R^d \to V$ where $V$ is a function like sector with $\zeta := \min\{ \alpha >0 : V_\alpha \neq 0 \}$ and $\gamma >  \zeta > 0$. Then $\hat{F}_\gamma(f) \in \CD^\gamma(V)$ and we have the explicit bounds for compact $\fK \subset \R^d$
    \begin{equs} \label{eq:explicit_bounds_composition}
    \$ \hat{F}_\gamma(f) \$_{\gamma; \fK} &\lesssim \left(\sup_{k \leq \frac{\gamma}{\zeta}} \norm{D^k F}_\infty  \right) (1 \vee \$Z \$_{\gamma; \fK})^{\frac{\gamma}{\zeta}}(1 \vee \$ f \$_{\gamma; \fK})^{\frac{\gamma}{\zeta}}\, \\ \label{eq:explicit_bounds_composition_difference}
    \$ \hat{F}_\gamma(f) - \hat{F}_\gamma(g) \$_{\gamma; \fK} &\lesssim \left(\sup_{k \leq \frac{\gamma}{\zeta} + 1 } \norm{D^k F}_\infty  \right)  (1 \vee \$Z \$_{\gamma; \fK})^{\frac{\gamma}{\zeta}}(1 \vee ( \$ f \$_{\gamma; \fK} + \$ g \$_{\gamma;\fK}) )^{\frac{\gamma}{\zeta}} \$ f - g \$_{\gamma; \fK} \, . 
    \end{equs}
    Consequently, we have for $F,G: \R \to \R$ the bound and 
    \begin{equs}\label{eq:bound_difference_of_difference}
    & \$ \hat{F}_\gamma(f) - \hat{G}_\gamma(g) \$_{\gamma; \fK} \\ \lesssim & (1 \vee \$ Z \$_{\gamma; \fK})^{\frac{\gamma}{\zeta}} (1 \vee (\$ f \$_{\gamma; \fK} + \$ g \$_{\gamma; \fK}))^{\frac{\gamma}{\zeta}} \left( (1 \vee \left(\sup_{k \leq \frac{\gamma}{\zeta} + 1 } \norm{D^k F}_\infty  \right)) \$ f- g \$_{\gamma; \fK} +  \sup_{k \leq \frac{\gamma}{\zeta}}\norm{D^k(F- G)}_{\infty} \right)  \, . 
    \end{equs}
    Analougous bounds hold for when considering weighted modelled distribution spaces $\CD^{\gamma,\eta}$ with $ \gamma \geq \eta  $. 
\end{lemma}
\begin{proof}
    Following line by line the proof in \cite[Theorem 4.16]{hairer2014theory} (respectively \cite[Proposition 6.13]{hairer2014theory}) and keeping track of all the implicit proportionality constants therein yields the bounds for $f,g \in \CD^\gamma$ (respectively in $\CD^{\gamma, \eta}$). 
\end{proof}

%\printbibliography

\paragraph{Acknowledgements} SE is funded by the European Research Council (ERC) under
the European Union's Horizon 2020 research and innovation programme (Grant agreement No.
101045082), and by the Deutsche Forschungsgemeinschaft (DFG, German Research Foundation)
under Germany's Excellence Strategy EXC 2044-390685587, Mathematics Münster: Dynamics-
Geometry-Structure. The authors thank Sophie Mildenberger and Hendrik Weber for helpful discussions. 

\paragraph{AI usage statement.} OpenAI's ChatGPT 5.6 Sol was used during the initial stages of this work. In particular the main idea of having a uniform existence time for constant initial data being leveraged into the a priori bound via the comparison principle was revelead during a chat session.  The proof in its final form was developed by the authors, who take full responsibility for all claims and arguments made in the paper.

\bibliographystyle{myalpha} % We choose the "plain" reference style
\bibliography{refs} % Entries are in the refs.bib file

%\printbibliography

\vspace{2em}
\noindent
\small
(S.~Esquivel) \textsc{Institut für Analysis und Numerik, Universität Münster, Münster,  Germany}\\
\hspace*{1.5em}\textit{Email address}: \texttt{salvador.esquivel@uni-muenster.de}

\vspace{2em}
\noindent
(F.~Schwarzfischer) \textsc{Department of Mathematics, Technical University of Munich, Garching b. München,  Germany}\\
\hspace*{1.5em}\textit{Email address}: \texttt{schwarzfischer.felix@web.de}

\end{document}